\documentclass[11pt,a4paper]{article}

\RequirePackage{fix-cm}
\usepackage[misc]{ifsym} 
\usepackage{caption}
\usepackage{float}
\usepackage{xcolor}
\usepackage{bm}
\usepackage[linesnumbered,ruled,vlined,algo2e]{algorithm2e} 
\usepackage[utf8]{inputenc}

\usepackage{amsmath,amssymb,amsfonts,amsthm} 
\usepackage{graphicx}
\usepackage{cite}
\usepackage[margin=1in]{geometry}

	\theoremstyle{plain}
	\newtheorem{theorem}{Theorem}[section]
	\newtheorem{lemma}[theorem]{Lemma}
	
	\newtheorem{corollary}[theorem]{Corollary}
	
	\theoremstyle{definition}

	\newenvironment{keywords}{%
		\par\noindent\textbf{Keywords:} 
	}{%
		\par\vspace{1em}
	}

	\newcommand{\revq}[1]{{\color{black}{#1}}}
	
\begin{document}

\markboth{Caixia Kou, Feifei Gao, Yu-Hong Dai}{Mini-batch Stochastic Conjugate Gradient Algorithms with Minimal Variance}

\title{MINI-BATCH STOCHASTIC CONJUGATE GRADIENT ALGORITHMS WITH MINIMAL VARIANCE\footnote{This research is supported by the Beijing Natural Science Foundation (Grants No. Z220004), the Strategic Priority Research Program of Chinese Academy of Sciences (No. XDA27010101), the Chinese NSF grant (No.  12021001) and the Fundamental Research Funds for the Central Universities(No. 2023ZCJH02)}}

\author{Caixia Kou%
	\thanks{Key Laboratory of Mathematics and Information Networks (Beijing University of Posts and Telecommunications), Ministry of Education; School of Mathematical Sciences, Beijing University of Posts and Telecommunications, Beijing, 100876, China. Email: koucx@bupt.edu.cn}%
	\and
	Feifei Gao%
	\thanks{School of Mathematical Sciences, Beijing University of Posts and Telecommunications, Beijing, 100876, China; Department of Statistics, Zhejiang Gongshang University Hangzhou College of Commerce, Hangzhou, 310018, China}%
	\and
	Yu-Hong Dai%
	\thanks{*Corresponding author. LSEC, ICMSEC, AMSS, Chinese Academy of Sciences, Beijing 100190, China; School of Mathematical Sciences, University of Chinese Academy of Sciences, Beijing 100049, China}%
}
\date{} 
\maketitle

\begin{abstract}
Stochastic gradient descent (SGD) methods have gained widespread popularity for solving large-scale optimization problems. However, the inherent variance in SGD often leads to slow convergence rates. We introduce a family of unbiased stochastic gradient estimators that encompasses existing estimators from the literature and identify a gradient estimator that not only maintains unbiasedness but also achieves minimal variance. Compared with the existing estimator used in SGD algorithms,
the proposed estimator demonstrates a significant reduction in variance.
By utilizing this stochastic gradient estimator to approximate the full gradient, 
we propose two mini-batch stochastic conjugate gradient algorithms with minimal variance. Under the assumptions of strong convexity and smoothness on the objective function, we prove that the two algorithms achieve linear convergence rates. Numerical experiments validate the effectiveness of the proposed gradient estimator in reducing variance and demonstrate that the two stochastic conjugate gradient algorithms exhibit accelerated convergence rates and enhanced stability.
\end{abstract}

\noindent\textbf{2020 Mathematics Subject Classification:} 49M37, 90C25

\begin{keywords}
Stochastic gradient descent (SGD), Minimal variance, Stochastic conjugate gradient (SCG) , Stochastic gradient estimator 
\end{keywords}

\section{Introduction}
\label{intro}
Consider the empirical risk minimization (ERM) problem
\begin{equation}
	\min\limits_{\omega \in \mathbb{R}^{d}} f(\omega) = \frac{1}{n} \sum_{i=1}^n f_i(\omega),
	\label{1:problem}
\end{equation}
where $\omega\in\mathbb{R}^d$  represents the decision vector, $f_i(\omega):\mathbb{R}^{d}\to \mathbb{R} $ is the loss function for the $i$-th sample, $n$ is the number of samples. This problem is widely arised from machine learning for building and training models to minimize some loss on a training dataset.

In the case when the number of samples $n$ is extremely large, the calculating of full gradients $\nabla f(\omega)$ becomes expensive. Stochastic Gradient Descent (SGD),  proposed in the seminal work \cite{Robbins1951SGD}, randomly generates one sample of the loss function to update the iteration. As the gradient direction oscillates in the SGD method, the mini-batch SGD was introduced in \cite{Shalev2007Minibatch,zhangt2014Minibatch2}, which is to select some subset of samples instead of a single sample at each iteration. This technique is able to reduce gradient variance and hence increases the stability of the method.  The stochastic average gradient (SAG) method \cite{Roux2012sag} needs to store the gradient of each sample and utilizes the average of historical gradients to reduce variance. The SAG method achieves a linear convergence rate in strongly convex problems, which is much faster than SGD. The stochastic variance reduced gradient (SVRG) \cite{johnson2013accelerating} does not need to maintain all gradients in memory, but performs gradient replacement every $m$ iterations and achieves the linear convergence as well.  The SAGA method \cite{defazio2014saga}, like SAG, needs to store the gradients of all samples, yet retains unbiasedness.  It is worth noting that, although the variance produced by the gradient estimation in SAG is $1/n^2$  times that of SAGA, this reduction in variance comes at the expense of introducing a non-zero bias \cite{defazio2014saga}.  Other variants of variance reduction algorithms, as well as through adaptive learning rates, can be found in \cite{Nguyen2017sarah,Shalev2013SDCA,Duchi2011adagrad,Zeiler2012AdaDelta,Tieleman2012rmsprop,Kingma2015adam,Tan2016svrgbb,yutengteng2021SVRGTR,yutengteng2023SVRGDBB}.  For more SGD algorithms, we refer to \cite{bottou2010large,sun2019review, Tian2023review, Lan2020review} and the references therein.

This leads us to focus on how to estimate the gradients in the best way. We believe that an ideal gradient estimator should be unbiased and  exhibit low variance.
Therefore, this paper mainly focuses on designing unbiased stochastic gradient estimators that minimize variance, and applies them to stochastic conjugate gradient algorithms, aligning with our research interests.

Recalling the classical conjugate gradient method, it is widely used for large-scale optimization problems due to its fast convergence and low storage requirement. Well-known formulas for the conjugate gradient parameter are Fletcher–Reeves (FR) \cite{fletcher1964function}, Polak–Ribière–Polyak (PRP) \cite{polak1969note}\cite{polyak1969conjugate}, Hestenes–Stiefel (HS) \cite{hestenes1952methods}, and Dai–Yuan (DY) \cite{dai1999nonlinear} ones. A range of hybrid conjugate gradient methods, such as TAS \cite{touati1990efficient}, PRP-FR \cite{hu1991global}, and GN \cite{gilbert1992global}, have been developed as well. These hybrid methods combine the properties of standard conjugate gradient methods to acquire new characteristics, facilitating rapid convergence to the solution \cite{andrei2020nonlinear}.
By exploring the second-order information and analyzing the relationship between  conjugate gradient directions and  quasi-Newton directions, more efficient conjugate gradient methods have been developed in Dai and Kou \cite{dai2013nonlinear} and Hager and Zhang \cite{hager2005new,hager2013limited}. For more details on conjugate gradient methods, we refer to \cite{andrei2020nonlinear} and \cite{Dai2020nonlinear}.

Applying the conjugate gradient algorithms to solve ERM problem \eqref{1:problem}, 
two stochastic conjugate gradient algorithms, 
namely SCGA \cite{kou2022mini} and CGVR \cite{jin2018stochastic}, have been proposed. It is found that these stochastic conjugate gradient methods demonstrate a faster decrease in the objective function value compared to SGD. 
Motivated by these findings, in the second part of this paper, we introduce two stochastic conjugate gradient algorithms with minimal variance by integrating the new proposed gradient estimator.


The rest of this paper organized is as follows:  Section 2 proposes a  new family of stochastic gradient estimators with unbiasedness and a new  stochastic gradient estimator with minimal variance. Section 3 describes the details of the two stochastic conjugate gradient algorithms  with  minimal variance. The convergence of the proposed algorithms is proved in Section 4. Section 5 shows the results of numerical experiments and Section 6 concludes this paper.

\section{A Stochastic Gradient Estimator with Minimal Variance}
\label{sec:1}
To begin with, we review the stochastic gradient estimators employed in  stochastic gradient algorithms, SAGA and SVRG, which incorporate variance reduction techniques. We simplify these estimators into a more streamlined form
\begin{equation}
	g_k=\nabla f_{S_k}(\omega_k)-\nabla f_{S_k}(\phi_k)+\frac{1}{n}\sum_{i=1}^{n}\nabla f_i(\phi_k).
	\label{2:SAGA-SVRG}
\end{equation}

In the above equation, it involves two iterative points: one is the current iterate, denoted by $\omega_k\in\mathbb{R}^d$, which represents the current state of the optimization process. The other is a reference point, designated as $\phi_k\in\mathbb{R}^d$, which serves as a benchmark or comparison point in the iterative procedure. The choices of this reference point varies significantly across different stochastic gradient descent algorithms with variance reduction, particularly in SAGA and SVRG algorithms.
In the $k$-th iteration, $S_k$ is a mini-batch of samples randomly extracted from all available samples.
The stochastic gradient estimator $g_k$ in \eqref{2:SAGA-SVRG} is obtained by correcting the stochastic gradient on $S_k$ at the current iterate $\omega_k$ with the difference between the stochastic gradient on $S_k$ at the reference point $\phi_k$ and the full gradient at $\phi_k$.
The purpose of this correction is to reduce the variance of the stochastic gradient at the current iterate $\omega_k$. By incorporating information from the full gradient at the reference point $\phi_k$, the estimator  $g_k$ becomes a more accurate approximation of the full gradient, leading to faster convergence and improved performance of the stochastic gradient descent  algorithm in practice.

For convenience in subsequent analysis, we define, for each $j \in S_k$,
\begin{equation}
	\mathbf{X}_j=\nabla f_j(\omega_k),~\mathbf{Y}_j=\nabla f_j(\phi_k),
	\label{2:varible_define1}
\end{equation}
\begin{eqnarray}\nonumber
	\bar{\mathbf{X}}=\nabla f_{S_k}(\omega_k),~\bar{\mathbf{Y}}=\nabla f_{S_k}(\phi_k),~ E(\bar{\mathbf{Y}})=\frac{1}{n}\sum_{i=1}^{n}\nabla f_i(\phi_k).
	\label{2:varible_define2}
\end{eqnarray}
In addition, $\bar{\mathbf{X}}$ and $\bar{\mathbf{Y}}$ represent the mean of $\mathbf{X}_j$ and $\mathbf{Y}_j$  on the subset $S_k$, respectively. Obviously, the expectation of $\bar{\mathbf{Y}}$ is precisely the full gradient at $\phi_k$.

Using the above definitions, we now introduce a new family of stochastic gradient estimators
\begin{equation}
	g_k(\Lambda)=\bar{\mathbf{X}}-\Lambda(\bar{\mathbf{Y}}-\mathbb{E}(\bar{\mathbf{Y}})),
	\label{family_define}
\end{equation}
with the diagonal matrix $\Lambda=\text{diag}(\lambda_1,\cdots,\lambda_d)$  linearly combining $\bar{\mathbf{X}}$ and $\mathbb{E}(\bar{\mathbf{Y}})-\bar{\mathbf{Y}}$.
\vspace{0.3cm}\\
\textbf{Remarks:}
\begin{itemize}
	\item[1.] The estimator utilized in SVRG and SAGA algorithms is a particular member of this family. In fact,  when $\Lambda$ is set to the identity matrix $I$, the estimator  $g_k(I)$ corresponds to the estimator of SVRG and SAGA given in \eqref{2:SAGA-SVRG}. 
	\item[2.] This family of stochastic gradient estimators is  unbiased because the expected value of the estimator equals the true gradient, i.e.,  
	$
	\mathbb{E}(g_k(\Lambda))=\mathbb{E}(\bar{\mathbf{X}})=\nabla f(\omega_k).
	$
	\item[3.] Among the family, we can determine the one with minimal variance. 
\end{itemize}

In fact, calculating the variance of the family by 
\begin{equation}\nonumber
	\begin{aligned}
		Var(g_k(\Lambda)) =Var(\bar{\mathbf{X}})+\Lambda^2 Var(\bar{\mathbf{Y}})-2 \Lambda Cov(\bar{\mathbf{X}},\bar{\mathbf{Y}}),
		\label{2:variance}
	\end{aligned}
\end{equation}
and minimizing the above variance, we obtain the optimal value of $\Lambda$ as
\begin{equation}
	\Lambda^*=\frac{Cov(\bar{\mathbf{X}}, \bar{\mathbf{Y}})}{Var(\bar{\mathbf{Y}})}.
	\label{3:gamma_star}
\end{equation}
Note that $Cov(\bar{\mathbf{X}}, \bar{\mathbf{Y}})$ and $Var(\bar{\mathbf{Y}})$ are  diagonal matrices, since it is reasonable that we assume  the same components of vectors $\bar{\mathbf{X}}$ and $\bar{\mathbf{Y}}$ are correlated, while different components of them are independent.

%
%
	Substituting $\Lambda^*$ into the estimators' family \eqref{family_define}, we propose a new stochastic gradient estimator with minimal variance:
	\begin{equation}
		g_k(\Lambda^*)=\bar{\mathbf{X}}-\Lambda^*(\bar{\mathbf{Y}}-\mathbb{E}(\bar{\mathbf{Y}})).
		\label{3:sgmv}
	\end{equation}
	
	\textbf{Remarks:}
	\begin{itemize}
		\item[1.] 
		The stochastic gradient estimator $g_k(\Lambda^*)$ achieves the minimal variance within the family defined by \eqref{family_define}, and this variance is also smaller than that of SVRG and SAGA.
		i.e. $$Var(g_k(\Lambda^*))\leq Var(g_k(I)).$$
		
		\item[2.]  The minimal variance can be written as $$Var(g_k(\Lambda^*))=Var(\bar{\mathbf{X}})(I-\rho_{\bar{\mathbf{X}}\bar{\mathbf{Y}}}^2),$$ where $\rho_{\bar{\mathbf{X}}\bar{\mathbf{Y}}}$ is a diagonal matrix composed of the correlation coefficients of each component of $\bar{\mathbf{X}}$ and $\bar{\mathbf{Y}}$, with components lying within the interval [−1,1].
		Through the above equation, we have the following facts:
		As the iteration point gets closer to the optimal solution, points $\omega_k$​ and $\phi_k$ will converge, resulting in a near-linear relationship between $\bar{\mathbf{X}}$ and $\bar{\mathbf{Y}}$, i.e., $\rho_{\bar{\mathbf{X}}\bar{\mathbf{Y}}}$ approaches $I$. Consequently, this leads to the variance $Var(g_k(\Lambda^*))$ approaching 0. 
		%
		On the other hand, when $\omega_k$ and $\phi_k$ converge to the optimal solution $\omega^*$, then $\bar{\mathbf{Y}}$ approaches $\bar{\mathbf{X}}$. Consequently, $Cov(\bar{\mathbf{X}}, \bar{\mathbf{Y}})$ approaches $Var(\bar{\mathbf{Y}})$. According to equation \eqref{3:gamma_star}, $\Lambda^*$ approaches $I$. These facts will also be observed in the numerical results of Section 5.1.
	\end{itemize}
	%
	
	To apply this new estimator given in \eqref{3:sgmv} to SGD algorithms, we need to consider how to calculate $\Lambda^*$.
	Based on the definitions of $Cov(\bar{\mathbf{X}},\bar{\mathbf{Y}})$ and $Var(\bar{\mathbf{Y}})$,
	it is not difficult to find that calculating them requires the full gradient information $\mathbb{E(\bar{\mathbf{X}})}$ and $\mathbb{E(\bar{\mathbf{Y}})}$ corresponding to these two mini-batch stochastic gradients $\bar{\mathbf{X}}$ and $\bar{\mathbf{Y}}$. However, in the $k$-th iteration of SGD algorithms, we can only obtain the gradient of each sample on the mini-batch $S_k$, but cannot obtain the full gradient. Therefore, we need to approximate $Cov(\bar{\mathbf{X}},\bar{\mathbf{Y}})$ and $Var(\bar{\mathbf{Y}})$ with the known samples $S_k$.
	
	Suppose the sample size of a mini-batch $S_k$ at each iteration is the same, denoted by $|S|$. Then, the digonal covariance matrix $Cov(\bar{\mathbf{X}},\bar{\mathbf{Y}})$ can be obtained
	\begin{equation}\nonumber
		\begin{aligned}
			Cov(\bar{\mathbf{X}},\bar{\mathbf{Y}})  &=Cov(\frac{1}{|S|}\sum_{j\in S_k} \mathbf{X}_j,\frac{1}{|S|}\sum_{j\in S_k} \mathbf{Y}_j)=\frac{1}{|S|^2}Cov(\sum_{j\in S_k} \mathbf{X}_j,\sum_{j\in S_k} \mathbf{Y}_j)\\
			&=\frac{1}{|S|^2}\sum_{j\in S_k} Cov(\mathbf{X}_j,\mathbf{Y}_j)=\frac{1}{|S|} Cov(\mathbf{X}_j,\mathbf{Y}_j)
			\approx \frac{1}{|S|}s_\mathbf{XY}.
		\end{aligned}
	\end{equation}
	The first equality is derived from the definitions of $\bar{\mathbf{X}}$ and $\bar{\mathbf{Y}}$. Subsequently, applying the properties of covariance leads to the second equality. Given that $\mathbf{X}_i$ and $\mathbf{Y}_j$ ($i\neq j$) are independent, it follows that $Cov(\mathbf{X}_i,\mathbf{Y}_j)=0$, thereby validating the third equality. For any $j \in S_k$, the covariances $Cov(\mathbf{X}_j,\mathbf{Y}_j)$ are uniform, meaning that $\sum_{j\in S_k}Cov(\mathbf{X}_j,\mathbf{Y}_j) = |S|Cov(\mathbf{X}_j,\mathbf{Y}_j)$, which justifies the fourth equation. In the final approximate equality, $Cov(\mathbf{X}_j, \mathbf{Y}_j)$ is approximated by $s_\mathbf{XY}$. The matrix $s_\mathbf{XY}$ is diagonal, with its $r$-th diagonal element $s_\mathbf{XY}^{(r)}$ calculated using the $r$-th component $\mathbf{X}^{(r)}_j$ and $\mathbf{Y}^{(r)}_j$ of the stochastic gradients $\mathbf{X}_j$ and $\mathbf{Y}_j$. The specifics of this calculation are outlined below:
	\begin{equation}
		\begin{aligned}
			s_\mathbf{XY}^{(r)}&=\frac{1}{|S|-1}\sum_{j\in S_k}(\mathbf{X}^{(r)}_j-\bar{\mathbf{X}}^{(r)})(\mathbf{Y}^{(r)}_j-\bar{\mathbf{Y}}^{(r)}),
		\end{aligned}
		\label{s_xy}
	\end{equation}
	where $\bar{\mathbf{X}}^{(r)}$ and $\bar{\mathbf{Y}}^{(r)}$ is the mean of $\mathbf{X}^{(r)}_j$ and $\mathbf{Y}^{(r)}_j$ on $S_k$.
	
	The estimation process for $Var(\bar{\mathbf{Y}})$ unfolds as follows:
	\begin{equation}\nonumber
		\begin{aligned}
			Var(\bar{\mathbf{Y}})  &=Var(\frac{1}{|S|}\sum_{j\in S_k} \mathbf{Y}_j)=\frac{1}{|S|^2}Var(\sum_{j\in S_k} \mathbf{Y}_j)\\
			&=\frac{1}{|S|^2}\sum_{j\in S_k} Var(\mathbf{Y}_j) = \frac{1}{|S|}Var(\mathbf{Y}_j)\approx \frac{1}{|S|}s_\mathbf{Y}^2.\\
		\end{aligned}
	\end{equation}
	Analogous to the estimation of $Cov(\bar{\mathbf{X}},\bar{\mathbf{Y}})$, the first four equations are derived from the definition of $Var(\bar{\mathbf{Y}})$, the properties of variance, the independence of the samples in $S_k$, and the property of equal variance for gradients across different samples, respectively. In the final approximate equality, we use the sample variance matrix $s_\mathbf{Y}^2$ as an estimate for $Var(\mathbf{Y}_j)$. The matrix $s_\mathbf{Y}^2$ is diagonal, with its $r$-th diagonal component $s_{\mathbf{Y}^{(r)}}^2$ calculated as follows:
	\begin{equation}
		\begin{aligned}
			s_{\mathbf{Y}^{(r)}}^2&=\frac{1}{|S|-1}\sum_{j\in S_k}(\mathbf{Y}^{(r)}_j-\bar{\mathbf{Y}}^{(r)})^2,
		\end{aligned}
		\label{s_y}
	\end{equation}
	
	In summary, the actual value of $\Lambda^*$ is calculated by
	\begin{equation}
		\tilde{\Lambda}^* = \frac{s_\mathbf{XY}}{s_\mathbf{Y}^2},
		\label{3:gamma_star_true}
	\end{equation}
	where each component of $s_\mathbf{XY}$ and $s_\mathbf{Y}^2$ is given in \eqref{s_xy} and \eqref{s_y} respectively.
	

To ensure the unbiasedness of the proposed estimator, theoretically, the calculation of $\tilde{\Lambda}^*$ in equation \eqref{3:gamma_star_true} requires a mini-batch samples $\widetilde{S}_k$ that is independent of $S_k$, ensuring $E(g_k) = \nabla f(\omega_k)$. However, to save computational costs, in numerical experiments we still use $S_k$.

By integrating the new estimator into existing stochastic gradient descent algorithms, such as SAGA and SVRG, or stochastic conjugate gradient algorithms, like SCGA and CGVR, we can devise novel algorithms that exhibit minimal stochastic gradient variance. In light of our research endeavors, our primary focus in the subsequent section will be on improving two prominent stochastic conjugate gradient algorithms.

\section{Two Stochastic Conjugate Gradient Algorithms}
\label{sec:11}
In this section, we introduce two stochastic conjugate gradient algorithms by incorporating the new stochastic gradient estimator into SCGA \cite{kou2022mini} and CGVR \cite{jin2018stochastic}. Both of these are stochastic conjugate gradient algorithms, with the first employing the stochastic gradient estimator of the SAGA algorithm and the second utilizing the stochastic gradient estimator of the SVRG algorithm. We will now elaborate on these two algorithms.
\subsection{SCGA Algorithm with Minimal Variance }
\label{sec:21}
By incorporating the new stochastic gradient estimator into the SCGA algorithm, we introduce an improved version, designated as \revq{SCGA–MV}. This algorithm comprises two fundamental components: an initialization phase and an iterative process.

In the initialization phase, for each sample $i=1,2,...,n$\revq{,} we compute the gradient $\nabla f_i(\omega_{0})$ at the initial iteration point $\omega_0$  and store them in the matrix 
$$\nabla f(\omega_{[0]}) =( \nabla f_1(\omega_{[0]}),\nabla f_2(\omega_{[0]}),...,\nabla f_n(\omega_{[0]})).$$ 

Then, we compute the full gradient $\mu_0$ at $\omega_0$. After obtaining the gradient $\nabla f_j(\omega_k)$ for all $j \in S_k$ at the $k$-th iteration, we update the corresponding entry $\nabla f_j(\omega_{[k-1]})$ of the stored matrix $\nabla f(\omega_{[k-1]})$ with $\nabla f_j(\omega_k)$ for each $j \in S_k$, while other
entries remain unchanged at the last step of the iteration, as detailed in equation (\ref{3:scga_full_gradient}). It should be noted that $\nabla f(\omega_{[k]})$ represents the last stored full gradient matrix at the $k$-th iteration, which corresponds to the virtual iteration point, denoted as $\omega_{[k]}$.
Additionally, we set the initial gradient $g_0=\mu_0$ and the initial direction $d_0=-g_0$ to compute the first iteration point
$\omega_1$. In the last step of the initialization phase, we also randomly sample a mini-batch  $S_0$ to calculate the first step size.

In the iteration process, we first calculate the step size $\alpha_{k-1}$ which satisfies the following strong Wolfe conditions
\begin{equation}
f_{S_{k}}(\omega_{k} + \alpha_{k} d_{k}) \leq f_{S_{k}}(\omega_{k}) +\sigma_1 \alpha_{k} g_k^T d_{k},
\label{3:wolfe_1}
\end{equation}
\begin{equation}
|g_{k+1}^T d_{k}| \leq -\sigma_2 g_{k}^T d_{k},
\label{3:wolfe_2}
\end{equation}
where $0 < \sigma_1< \sigma_2 < 1$. Then, it is straightforward to obtain the next iteration point $\omega_k=\omega_{k-1}+\alpha_{k-1}d_{k-1}$.

Next, we randomly extract a mini-batch of samples $S_k$ to determine the next search direction. The process of updating the search direction $d_{k}$  primarily involves the calculation of $\beta_k$ and $g_k$.
For the selection of the $\beta_k$, \revq{SCGA–MV} employs the PRP-FR \cite{hu1991global}, which is implemented as shown in the formula (\ref{3:beta_prp_fr}). This method outperforms other hybrid conjugate gradient methods \cite{andrei2020nonlinear}, such as TAS \cite{touati1990efficient} and GN \cite{gilbert1992global}. Moreover, the upper bound of $\beta_k^{PRP-FR}$ is $\beta_k^{FR}$, ensuring that the proof of convergence in Section 4 holds.
\begin{equation}
\beta_k^{PRP-FR}=max \{0,min \{\beta^{PRP},\beta^{FR}\}\},
\label{3:beta_prp_fr}
\end{equation}
where 
\begin{equation}\nonumber
\beta_k^{PRP}=\frac{g_k^T(g_k-g_{k-1})}{\|g_{k-1}\|^2}, \ \ 
\beta_k^{FR}=\frac{\|g_k\|^2}{\|g_{k-1}\|^2}.
\label{3:beta_fr}
\end{equation}
For the computation of $g_k$, we use the new stochastic gradient estimator (\ref{3:sgmv})  which is proposed in Section 2. Inside, the reference point $\phi_k$ takes the value of $\omega_{[k-1]}$, which is the same as that used in SCGA and SAGA.
We calculate the diagonal matrix $\Lambda^*$ by the following steps.
Firstly for each $j \in S_k$, compute the gradient $\nabla f_{j}(\omega_k)$ at the current iterate point $\omega_k$ and select the gradient $\nabla f_{j}(\omega_{[k-1]})$ from the last stored full gradient matrix $\nabla f(\omega_{[k-1]})$. Afterward, we let $\mathbf{X}_j = \nabla f_{j}(\omega_k)$ and $\mathbf{Y}_j = \nabla f_{j}(\omega_{[k-1]})$, and use the formula (\ref{3:gamma_star_true}) together with (\ref{s_xy}) and (\ref{s_y}) to compute  $\tilde{\Lambda}^*$.
Ultimately, the stochastic gradient estimator $g_k$ is computed with
\begin{equation}
	g_k = \nabla f_{S_k}(\omega_k)-\tilde{\Lambda}^*(\mu_{S_k} - \mu_{k-1}),
	\label{algorithm1:gk}
\end{equation}
where the stochastic gradient of the reference point $\omega_{[k-1]}$ in the mini-batch $S_k$ is 
\begin{equation}
\mu_{S_k} = \frac{1}{|S|}  \sum_{j\in S_k}\nabla f_{j}(\omega_{[k-1]}),
\label{musk}
\end{equation}
and the full gradient at $\omega_{[k-1]}$ is 
stored in $\mu_{k-1}$, defined by 
$	\mu_{k-1} = \frac{1}{n} \sum_{j=1}^{n}\nabla f_{j}(\omega_{[k-1]}).
$
At the end of each iteration, we update the full gradient matrix with
\begin{equation}
\begin{aligned}
	\begin{split}
		\nabla f(\omega_{[k]})= \left \{
		\begin{array}{ll}		
			\nabla f_j(\omega_k),  & \forall j \in S_k\\
			\nabla f_j(\omega_{[k-1]}),& \forall j \notin S_k\\
		\end{array}		
		\right.
	\end{split}	
	\label{3:scga_full_gradient}
\end{aligned}
\end{equation}
and update the new full gradient
\begin{equation}
\mu_k = \frac{1}{n}  \sum_{j=1}^{n} \nabla f_j(\omega_{[k]}).
\label{algorithm:muk}
\end{equation}

The main framework of \revq{SCGA–MV} is described as follows:\\
\begin{algorithm2e}[H]
	\textbf{Initialization:} Given $\omega_0 \in \mathbb{R}^d$, compute the full gradient matrix at initial iterate $\omega_0$ and store it:\\
	\For {$i = 1,2, \ldots, n$} 
	{Compute $\nabla f_i(\omega_0)$. \\
		Store $\nabla f(\omega_{[0]})\leftarrow \nabla f_i(\omega_0)$.}
	$\mu_0 =\frac{1}{n}  \sum_{i=1}^{n} \nabla f_i(\omega_{[0]})$.\\
	Set the initial gradient and the direction $g_0=\mu_0, d_0=-g_0$.\\
	Randomly sample a mini-batch $S_0$.\\
	\textbf{Iteration:}\\
	\For {$k = 1,2, \ldots$} 
	{
		Find the step size $\alpha_{k-1}$ satisfying (\ref{3:wolfe_1}) and (\ref{3:wolfe_2}).\\
		Update iterate $\omega_k = \omega_{k-1} + \alpha_{k-1}d_{k-1}$.\\
		Randomly sample a mini-batch $S_k$.\\		
		Compute $g_k$ using (\ref{algorithm1:gk}) together with \eqref{musk}.			
		\\
		Choose $\beta_k $ as $ \beta_k^{PRP-FR}$ given in (\ref{3:beta_prp_fr}).\\
		Determine $d_k = -g_k + \beta_k d_{k-1}$.\\
		Update $\nabla f(\omega_{[k]})$ using \eqref{3:scga_full_gradient}.\\
		Update the new full gradient $\mu_{k}$ by
		(\ref{algorithm:muk}).\\
	}
	\caption{\revq{SCGA Algorithm with Minimal Variance}}\label{Algorithm1}
\end{algorithm2e}
{
\subsection{CGVR Algorithm with Minimal Variance }
\label{sec:3}
By integrating the new stochastic gradient estimator into the CGVR algorithm, \revq{we propose an improved version with minimal variance, designated as CGVR–MV. Its iterative process }consists of nested inner and outer loops. 

In the outer loop iteration, we periodically update the full gradient $\mu_{l-1}$ at the iteration point $\omega_{l-1}$, which is the last iteration point $x_m$ of the previous inner loop and is also the initial point $x_0$ of the next inner loop.

In the inner loop iteration, we firstly find the new step size $\alpha_{k-1}$ with the inexact line search satisfying the strong Wolfe conditions (\ref{3:wolfe_1}) and (\ref{3:wolfe_2}). Then, we obtain a new iteration point by $x_k = x_{k-1} + \alpha_{k-1}d_{k-1}$. Subsequently, we use the new stochastic gradient estimator (\ref{3:sgmv}) proposed in Section 2 to calculate $g_k$, where $\phi_k$ takes the value of the initial point $x_0$ in the current inner loop. The detailed steps are as follows:
We randomly sample a mini-batch $S_k$, and for any $j \in S_k$ calculate gradients $\nabla f_{j}(x_k)$ and $\nabla f_{j}(x_0)$ at iteration points $x_k$ and $x_0$, respectively. Next, let $\mathbf{X}_j = \nabla f_{j}(x_k)$ and $\mathbf{Y}_j = \nabla f_{j}(x_0)$, and calculate $\tilde{\Lambda}^*$ using the formula (\ref{3:gamma_star_true}) together with (\ref{s_xy}) and (\ref{s_y}).
Based on $\tilde{\Lambda}^*$, we obtain the stochastic gradient with minimal variance
\begin{equation}
	g_k =\nabla f_{S_k}(x_k)-\tilde{\Lambda}^*(\nabla f_{S_k}(x_0)- \mu_{l-1}).
	\label{algorithm2_gk}
\end{equation}

The subsequent calculation of  $d_k$  using  $g_k$ remains consistent with \revq{SCGA–MV}, and therefore, we omit the repetition here.

The main framework of \revq{CGVR–MV} is as follows:\\
\begin{algorithm2e}[H]
\DontPrintSemicolon
Given $\omega_0 \in \mathbb{R}^d$, and randomly sample a mini-batch $S_0$.\\	
\For {$l = 1, 2, \ldots, T$} 
{
	Compute $\mu_{l-1} = \frac{1}{n}  \sum_{i=1}^{n} \nabla f_i(\omega_{l-1})$.\\	
	Update $x_0 = \omega_{l-1},g_0 = \mu_{l-1},d_0 = - g_0$.\\
	\For {$k = 1, 2, \ldots, m$} 
	{
		Find the step size $\alpha_{k-1}$ satisfying (\ref{3:wolfe_1}) and (\ref{3:wolfe_2}).\\
		Update iterate $x_k = x_{k-1} + \alpha_{k-1}d_{k-1}$.\\
		Randomly sample a mini-batch $S_k$.\\		
		Compute $g_k$ using (\ref{algorithm2_gk}).\\
		Choose $\beta_k$ as $ \beta_k^{PRP-FR}$ in (\ref{3:beta_prp_fr}).\\
		Update $d_k = -g_k+\beta_k d_{k-1}$.\\			
	}
	$\omega_l = x_m$.\\
}
\caption{\revq{CGVR Algorithm with Minimal Variance}}\label{Algorithm2}
\end{algorithm2e}

\revq{Compared with the proposed SCGA with minimal variance, CGVR–MV has mainly the following differences. First,  it calculates the full gradient once in each outer loop, unlike SCGA–MV, which stores the gradient of each individual sample. Hence it has the characteristics of large computation and small storage, while SCGA–MV has low computation and large storage. Second, both of the proposed stochastic conjugate gradient algorithms use the formulas (\ref{3:sgmv})  to calculate the stochastic gradient and obtain the  minimal variance,  but the choices of the reference point $\phi_k$ in (\ref{3:sgmv}) are different. CGVR-MV sets $\phi_k$ as the last iterative point $\omega_{l-1}$ of outer loops, while SCGA–MV uses the virtual point $\omega_{[k-1]}$. }

\section{Convergence Analysis}
\label{sec:22}
In this section, we provide convergence results for two proposed algorithms. The convergence analysis uses the same assumptions in SCGA \cite{kou2022mini} as follows:\\
\textbf{Assumption 1 ($\mu$-strong Convexity and $L$-smoothness)}
For each $i = 1, 2, \ldots, n$, the function $f_i$ is twice continuously differentiable  and
there exists constants $0 <\mu \leq L$ such that
\begin{equation}\nonumber
\mu I \preceq \nabla^2 f_i(w)\preceq LI, \ \ \forall w \in \mathbb{R}^d.
\label{4:assumptinon_1}
\end{equation}
\textbf{Assumption 2 (Lower and Upper Bounds of the Step Size)} \revq{Both of the proposed stochastic conjugate gradient algorithms} are implemented with the step size $\alpha_k$ satisfying $\alpha_{k} \in [\alpha_1, \alpha_2]$, where two constants $\alpha_1$ and $\alpha_2$  hold $0 \leq \alpha_1 \leq \alpha_2$.\\
\textbf{Assumption 3 (Upper Bound of Scalar $\beta_k$)} There exists a constant $0 < \beta < 1$ such that
\begin{equation}\nonumber
\beta_k \leq \frac{\|g_k\|^2}{\|g_{k-1}\|^2} \leq \beta.
\label{4:assumption_3}
\end{equation}
\begin{lemma}
Under Assumption 1, $\forall \omega \in \mathbb{R}^d$, we have
\begin{equation}\nonumber
	2\mu(f(\omega) - f(\omega^*)) \leq \|\nabla f(\omega)\|^2 \leq 2L(f(\omega) - f(\omega^*)),
	\label{4:lemma_1_formula}
\end{equation}
where $\omega^*$ is the unique minimizer.
\label{lemma1}
\end{lemma}
\begin{lemma} 
Consider \revq{the two proposed stochastic conjugate gradient algorithms}, where the step size $\alpha_k$ satisfies the strong Wolfe condition (\ref{3:wolfe_2}) with $0 < \sigma_2 < \frac{1}{2}$, and where $\beta_k$ satisfies $|\beta_k| \leq \beta_k^{FR}$. Then, the algorithms generate descent directions $d_k$ satisfying
\begin{equation}\nonumber
	-\frac{1}{1-\sigma_2} \leq \frac{g_k^T d_k}{\|g_k\|^2} \leq \frac{2\sigma_2-1}{1-\sigma_2}.
	\label{4:lemma_2_formula}
\end{equation} 
\label{lemma2}
\end{lemma}
The proof for Lemma \ref{lemma2} is provided in Lemma 3.1 of reference \cite{gilbert1992global}. This lemma establishes the lower and upper bounds for $\frac{g_k^T d_k}{\|g_k\|^2}$ when the parameter $\beta_k$ is suitably bounded in magnitude and $\alpha_k$ satisfies the strong Wolfe condition (\ref{3:wolfe_2}). Both the CGVR and SCGA algorithms rely on this same lemma to demonstrate convergence.

\revq{Next, the theoretical results of the two algorithms are presented. Based on Lemmas \ref{lemma1} and \ref{lemma2}, the linear convergence results of the proposed algorithms, SCGA-MV and CGVR-MV, are established in Theorems \ref{theorem_1} and \ref{theorem_2}, respectively.}
\begin{theorem}  
Suppose that Assumptions 1-3 hold. Let $\omega^*$ be the unique minimizer of $f$. Then $\forall k > 0$, we have
\begin{equation}\nonumber
E(f(\omega_k)-f({\omega^*})) \leq C\xi^k E(f(\omega_0)-f(\omega^*)),
\end{equation} 
where the parameters $C$ and $\xi$ are given by
\begin{equation}\nonumber
\begin{aligned}
	C& = 1+ \frac{2L \alpha_2 \sigma_1 \sigma_2}{(\xi-\beta)(1-\sigma_2)},\\
	\xi& = 1-2\mu \sigma_1 \alpha_1 <1,
\end{aligned}
\end{equation} 
assuming that we choose the lower bound $\alpha_1$ in the interval $\left(0, \frac{1-\beta}{2\mu\sigma_1}\right)$ for the step size.
\label{theorem_1}
\end{theorem}
\begin{proof}
Using the strong Wolfe condition (\ref{3:wolfe_1}), we have
\begin{equation}
f_{S_k}(\omega_{k+1})-f_{S_k}(\omega_{k}) \leq \sigma_1 \alpha_k g_k^T d_k.
\label{4:theorem_step1}
\end{equation} 
Taking expectations on both sides of (\ref{4:theorem_step1}), this yields
\begin{equation}\nonumber
E(f(\omega_{k+1})-f(\omega_{k})) \leq \sigma_1 E(\alpha_k g_k^T d_k).
\label{4:theorem_step2}
\end{equation}
Then, using the definition of $d_k$, we have
\begin{equation}\nonumber
E(f(\omega_{k+1})-f(\omega_{k})) \leq - \sigma_1 E(\alpha_k \|g_k\|^2)+\sigma_1 E(\alpha_k \beta_k g_k^T d_{k-1}).
\label{4:theorem_step3}
\end{equation}
Next, applying Assumption 2 and the strong Wolfe condition (\ref{3:wolfe_2}), we obtain that
\begin{equation}\nonumber
E(f(\omega_{k+1})-f(\omega_{k})) \leq - \sigma_1\alpha_1 E( \|g_k\|^2)-\sigma_1\sigma_2 E(\alpha_k \beta_k g_{k-1}^T d_{k-1}).
\label{4:theorem_step4}
\end{equation}
By Assumptions 2-3, and Lemma \ref{lemma2}, we have
\begin{equation}
E(f(\omega_{k+1})-f(\omega_{k})) \leq - \sigma_1\alpha_1 \|E( g_k)\|^2+\frac{\alpha_2 \sigma_1 \sigma_2 \beta }{1-\sigma_2} E(\| g_{k-1}\|^2).
\label{4:theorem_step5}
\end{equation}
Note that $g_k$ given in (\ref{3:sgmv}) is an unbiased estimator of the full gradient $\nabla f(\omega_k)$, i.e. 
\begin{equation}
E(g_k)=\nabla f(\omega_k).
\label{4:theorem_step6}
\end{equation}
Then, by (\ref{4:theorem_step6}) and Lemma \ref{lemma1}, we have
\begin{equation}
\|E(g_k)\|^2=\|\nabla f(\omega_k)\|^2 \ge 2 \mu (f(\omega_k)-f(\omega^*)).
\label{4:theorem_step7}
\end{equation}
On the other hand, Assumption 3 implies that 
\begin{equation}
E(\|g_k\|^2) \leq \beta E(\|g_{k-1}\|^2).
\label{4:theorem_step8}
\end{equation}
Then we unfold $g_{k-1}$ in (\ref{4:theorem_step8}) until reaching $g_0$ and use the definition of $g_0$, Lemma \ref{lemma1}, this yield
\begin{equation}
\begin{aligned}
	E(\|g_{k-1}\|^2) &\leq \beta^{k-1}E(\|g_0\|^2)\\
	&= \beta^{k-1}\|\nabla f(\omega_0)\|^2\\
	&\leq 2 \beta^{k-1}L(f(\omega_0)-f(\omega^*)).
	\label{4:theorem_step9}
\end{aligned}
\end{equation}
Together with (\ref{4:theorem_step7}) and (\ref{4:theorem_step9}), we continue to derive (\ref{4:theorem_step5}) as follows
\begin{equation}\nonumber
E(f(\omega_{k+1})-f(\omega_{k})) \leq -2\mu\sigma_1\alpha_1(f(\omega_k)-f(\omega^*))+\frac {2L\alpha_2 \sigma_1 \sigma_2 \beta^k}{1-\sigma_2}(f(\omega_0)-f(\omega^*)).
\label{4:theorem_step10} 
\end{equation}
Subtracting $f(\omega^*)$ from both sides of the above inequality, taking total expectations, and rearranging, this yields
\begin{equation}
E(f(\omega_{k+1})-f(\omega^*)) \leq (1-2\mu\sigma_1\alpha_1)E(f(\omega_k)-f(\omega^*))+\frac {2L\alpha_2 \sigma_1 \sigma_2 \beta^k}{1-\sigma_2}E(f(\omega_0)-f(\omega^*)).
\label{4:theorem_step11}
\end{equation}
For the convenience of discussion, we define
\begin{equation}
\Delta_{k+1}= E(f(\omega_{k+1})-f(\omega^*)),~\xi=1-2\mu\sigma_1\alpha_1,~\zeta=\frac {2L\alpha_2 \sigma_1 \sigma_2 }{1-\sigma_2}.
\label{4:theorem_step12}
\end{equation} 
We rewrite (\ref{4:theorem_step11}) as
\begin{equation}
\Delta_{k+1}\leq \xi \Delta_k+\zeta \Delta_0 \beta^k.
\label{4:theorem_step13}
\end{equation}  
Then the formula $\Delta_k$ in (\ref{4:theorem_step13}) is unfolded to the first term, we further obtain
\begin{equation}\nonumber
\begin{aligned}
	\Delta_{k+1}&\leq \xi^{k+1} \Delta_0 + \Delta_0 \zeta \sum_{i=0}^{k} \beta^i \xi^{k-i}\\
	&= \xi^{k+1}\Delta_0(1+\zeta \frac{1-(\frac{\beta}{\xi})^{k+1}}{\xi-\beta})\\
	&\leq \xi^{k+1}\Delta_0(1+\frac{\zeta}{\xi-\beta}).
	\label{4:theorem_step14}
\end{aligned}
\end{equation} 
Assuming that we choose the lower bound $\alpha_1$ in the interval $\left(0, \frac{1-\beta}{2\mu\sigma_1}\right)$ for the step size, then $\beta < \xi < 1$ ensures that the last inequality holds.

In the end, according to the above inequality and the definitions of $\Delta_k$, $\Delta_0$ and $\zeta$ in (\ref{4:theorem_step12}), we have
\begin{equation}\nonumber
E(f(\omega_{k+1})-f(\omega^*)) \leq C \xi^{k+1}E(f(\omega_0)-f(\omega^*)),
\label{4:theorem_step15}
\end{equation} 
where $C=1+ \frac{2L \alpha_2 \sigma_1 \sigma_2}{(\xi-\beta)(1-\sigma_2)}$ is a positive constant, and $\xi = 1-2\mu \sigma_1 \alpha_1 < 1$.\revq{The proof is completed.}
\end{proof}

Hence, SCGA–MV achieves a linear convergence rate for the strongly convex function. \revq{Since the convergence behavior of CGVR-MV within a single inner-loop iteration closely parallels that of SCGA-MV, we formalize the convergence properties of the CGVR-MV inner loop in Corollary 4.1, with its proof rigorously following the framework established in Theorem 4.1.} 
\begin{corollary}  
Suppose that Assumptions 1-3 hold. Let $\omega^*$ be the unique minimizer of $f$. Then, $\forall k > 0$, we have
\begin{equation}\nonumber
E(f(x_k)-f({\omega^*})) \leq C\xi^k E(f(x_0)-f(\omega^*)),
\label{4:theorem_concludtion}
\end{equation} 
where the parameters $C$ and $\xi$ are given by
\begin{equation}\nonumber
\begin{aligned}
	C& = 1+ \frac{2L \alpha_2 \sigma_1 \sigma_2}{(\xi-\beta)(1-\sigma_2)},\\
	\xi& = 1-2\mu \sigma_1 \alpha_1 <1,
\end{aligned}
\end{equation} 
assuming that we choose the lower bound $\alpha_1$ in the interval $\left(0, \frac{1-\beta}{2\mu\sigma_1}\right)$ for the step size.
\label{corollary_1}
\end{corollary}
\revq{Building upon the above corollary,  we proceed to derive a convergence theorem for the CGVR-MV algorithm.}
\begin{theorem}
Suppose that Assumption 1-3 hold. Let $\omega^*$ is the unique minimizer of $f$. Then, in \revq{CGVR–MV}, for all $l > 0$, we have
\begin{equation}\nonumber
E(f(\omega_{l})-f(\omega^*)) \leq  \tau^l E(f(\omega_0)-f(\omega^*)),
\end{equation} 
where the convergence rate is given as 
\begin{equation}
\tau=(1+ \frac{2L \alpha_2 \sigma_1 \sigma_2}{(1-2\mu \sigma_1 \alpha_1-\beta)(1-\sigma_2)})(1-2\mu \sigma_1 \alpha_1)^m<1,
\label{tau_define}
\end{equation}
assuming that we choose the lower bound $\alpha_1$ in the interval $\left(0, \frac{1-\beta}{2\mu\sigma_1}\right)$ for the step size and a sufficiently large $m$ to satisfy $\tau<1$.
\label{theorem_2}
\end{theorem}
\begin{proof}
According to Corollary \ref{corollary_1}, we have the following inequality for the inner loop of the $l$-th epoch:
\begin{equation}\nonumber
E(f(x_{m})-f(\omega^*)) \leq  \tau E(f(x_{0})-f(\omega^*)),
\end{equation}
where $\tau$ is given in (\ref{tau_define}).
\revq{Replacing the variables in the above inequality with those of the outer loop,} we obtain
\begin{equation}\nonumber
E(f(\omega_{l})-f(\omega^*)) \leq  \tau E(f(\omega_{l-1})-f(\omega^*)),
\end{equation}
\revq{which completes the proof. In other words, if $\alpha_1$ and $m$ satisfy (\ref{tau_define}), then CGVR–MV converges linearly in expectation.}
\end{proof}
 
\section{Numerical Experiments}
\label{sec:4}
\revq{In this section, employing fourteen datasets, we compare the two proposed stochastic conjugate gradient algorithms with SGD\cite{sun2019review}, SAGA\cite{defazio2014saga}, SVRG\cite{johnson2013accelerating}, SCGA\cite{kou2022mini}, and CGVR\cite{jin2018stochastic} algorithms, alongside the ridge regression model, to demonstrate the exceptional performance of the proposed algorithms.} A summary of these data sets is presented in Table \ref{tab1}. The datasets protein, quantum can be found in the KDD Cup 2004 website\footnote{http://osmot.cs.cornell.edu/kddcup}, and other datasets are available in LIBSVM\footnote{https://www.csie.ntu.edu.tw/~cjlin/libsvmtools/datasets/}.  All feature vectors of a9a and w8a are 0-1 variables, so we do not normalize them. \revq{The} remaining datasets are scaled to the range from -1 to 1. 
\begin{center}
\begin{table}[H]
\centering
\caption{Summary of data sets used in numerical experiments}
\label{tab1}       
\begin{tabular}{l c c l}
	\hline\noalign{\smallskip}
	dataset & d & n\\
	\noalign{\smallskip}\hline\noalign{\smallskip}
	a9a & 123 & 32561\\
	ijcnn1 & 22 & 49990\\
	protein & 74 & 145751\\
	quantum & 78 & 50000\\
	w8a & 300 & 49749\\
	covtype & 54 & 581012\\
	\revq{cifar10} & 3072 & 50000\\
	\revq{gisette} & 5000 & 6000\\
	YearPredictionMSD & 90 & 463715\\
	pyrim & 27 & 74\\
	bodyfat & 14 & 252\\
	triazines & 60 & 186\\
	eunite2001 & 16 & 336\\
	cpusmall & 12 & 8192\\
	\noalign{\smallskip}\hline
\end{tabular}
\end{table}
\end{center}
The ridge regression model is presented as follows:
\begin{equation}\nonumber
\begin{aligned}
\min\limits_{\omega} \frac{1}{n} \sum_{i=1}^{n} (y_i -x_i^T \omega)^2 +\gamma \|\omega\|^2,
\label{5:ridge}
\end{aligned}
\end{equation} 
where $x_i \in \mathbb{R}^d$ is denoted the feature vector of the $i$-th sample data, $y_i \in \mathbb{R}$  is denoted the actual value of the $i$-th sample data, and $\gamma$  is the regularization parameter.

\revq{The parameters of algorithms follow the experimental settings of SCGA and CGVR, as detailed in Table \ref{tab2}. Among the compared algorithms, SGD denotes mini-batch stochastic gradient descent with momentum acceleration (momentum coefficient = 0.9). For SGD, SAGA, and SVRG, we evaluated four constant step sizes: $10^{-2},10^{-3},10^{-4},10^{-5}$, and selected the optimal one for each algorithm.}
For fair comparison, all compared algorithms are implemented in Matlab 2018a, and the
experiments are conducted on a laptop with an Intel Core i5, 2.11 GHz processor and 16 GB of RAM running Windows 11 system. 
\begin{center}
\begin{table}[H]
\centering
\caption{\revq{Parameters in algorithms}}
\label{tab2}       
\begin{tabular}{l c c}
	\hline\noalign{\smallskip}
	Parameter & Description & Value \\
	\noalign{\smallskip}\hline\noalign{\smallskip}
	$\lvert S\rvert$ & the mini-batch size in all algorithms & $\sqrt{n}$ \\
	$\sigma_{1}$ & the parameter of line search in SCGA, SCGA-MV, CGVR and CGVR-MV& 
	$10^{-4}$ \\
	$\sigma_{2}$ & the parameter of line search in SCGA, SCGA-MV, CGVR and CGVR-MV& 0.1 \\
	$\gamma$ & the regularization parameter in the loss function & 
	$10^{-4}$ \\
	$m$ & the number of inner loop iterations in SVRG, CGVR, CGVR–MV & 50 \\
	\noalign{\smallskip}\hline
\end{tabular}
\vspace{-\baselineskip}
\end{table}
\end{center}

\subsection{Variance Comparison of Stochastic Gradient Estimators}
\label{sec:23}
On the $ijcnn1$ dataset, we generate 250 iteration points using the SVRG algorithm, with one outer loop executed every 50 inner loops. We compare the variances between the stochastic gradient estimator with the minimal variance $g(\tilde{\Lambda}^*)$ and the estimator $g(I)$ used in existing SGD algorithms.\\
\noindent 
\begin{minipage}{0.66\textwidth} 
\centering
\begin{center}
	\includegraphics[width=1\textwidth]{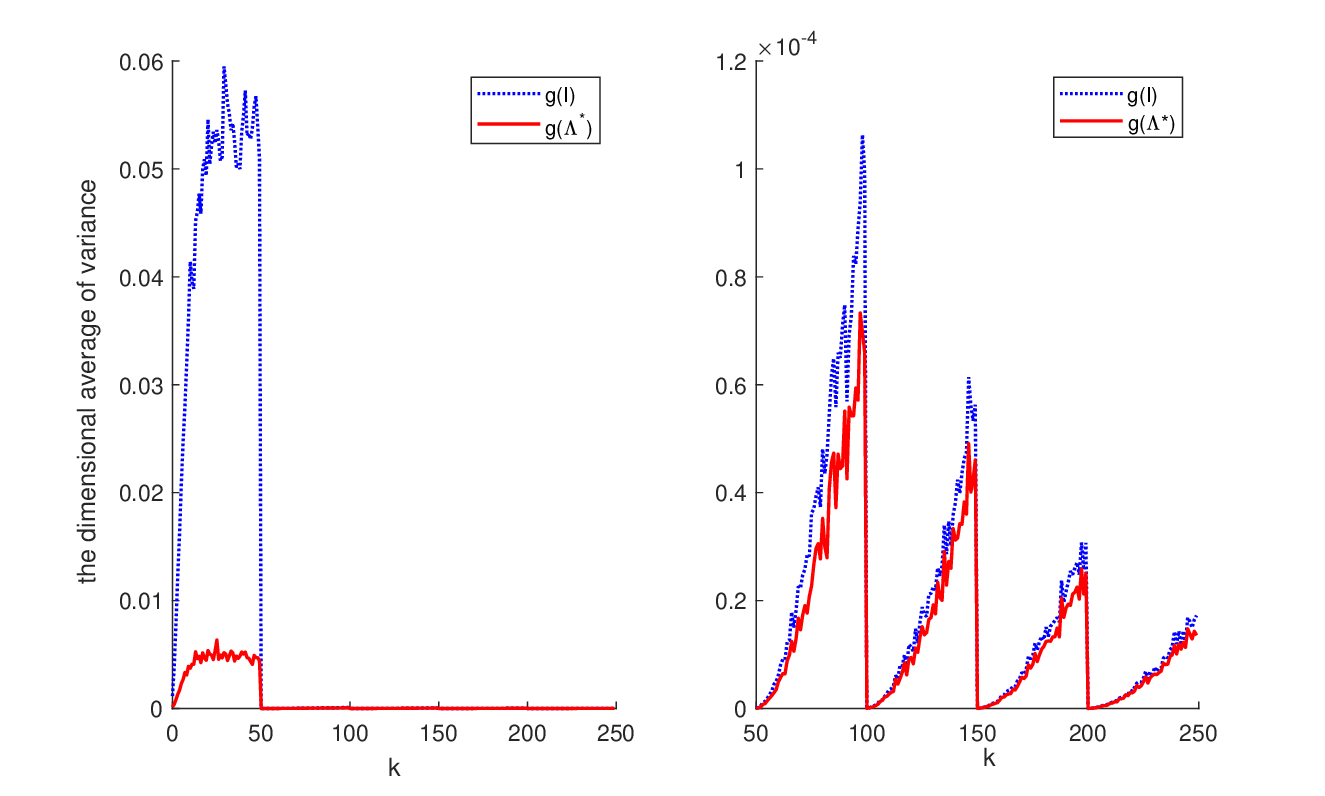}
\end{center}
\vspace{-0.4cm}
\captionsetup{margin={20pt,0pt}}
\captionof{figure}{Variance comparison of stochastic gradient estimators $g_k(\tilde{\Lambda}^*)$ and $g_k(I)$.}
\label{fig:1}       
\end{minipage}
\begin{minipage}{0.34\textwidth} 
\centering
\begin{center}
	\includegraphics[width=1\textwidth]{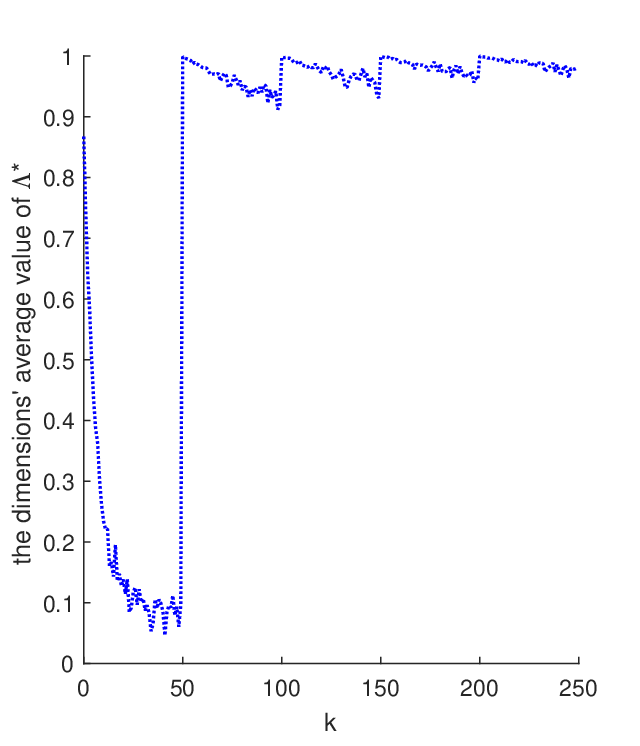}
\end{center}
\vspace{-0.4cm}
\captionsetup{margin={20pt,0pt}}
\captionof{figure}{Values of $\tilde{\Lambda}^*$ during iterations.}
\label{fig:2}       
\end{minipage}
\vspace{0.5cm}

The left figure in Fig \ref{fig:1} presents the variance comparison of stochastic gradient estimators $g_k(\tilde{\Lambda}^*)$ and $g_k(I)$ over 250 iterations. We observe that in the first 50 iterations, the variance of $g_k(\tilde{\Lambda}^*)$ is significantly lower than that of $g_k(I)$. However, in the last 200 iterations, the variance difference is less pronounced. This is because after the initial 50 iterations, the variances of both estimators decrease significantly, making direct comparison on the same vertical axis scale challenging. To facilitate a clearer comparison, the right figure specifically plots the variance comparison in the subsequent 200 iterations. The right figure in Fig \ref{fig:1} also shows that the stochastic gradient estimator $g_k(\tilde{\Lambda}^*)$ maintains a lower variance throughout the last 200 iterations.

In addition, Fig \ref{fig:2} illustrates that near the optimal solution, the value of $\Lambda^*$ converges to $I$, which is consistent with the theoretical analysis results in Section 2.

\subsection{Numerical Comparison of Proposed Algorithms}
\label{sec:24}
\begin{figure}[H]
	\centering
	\includegraphics[width=0.95\textwidth]{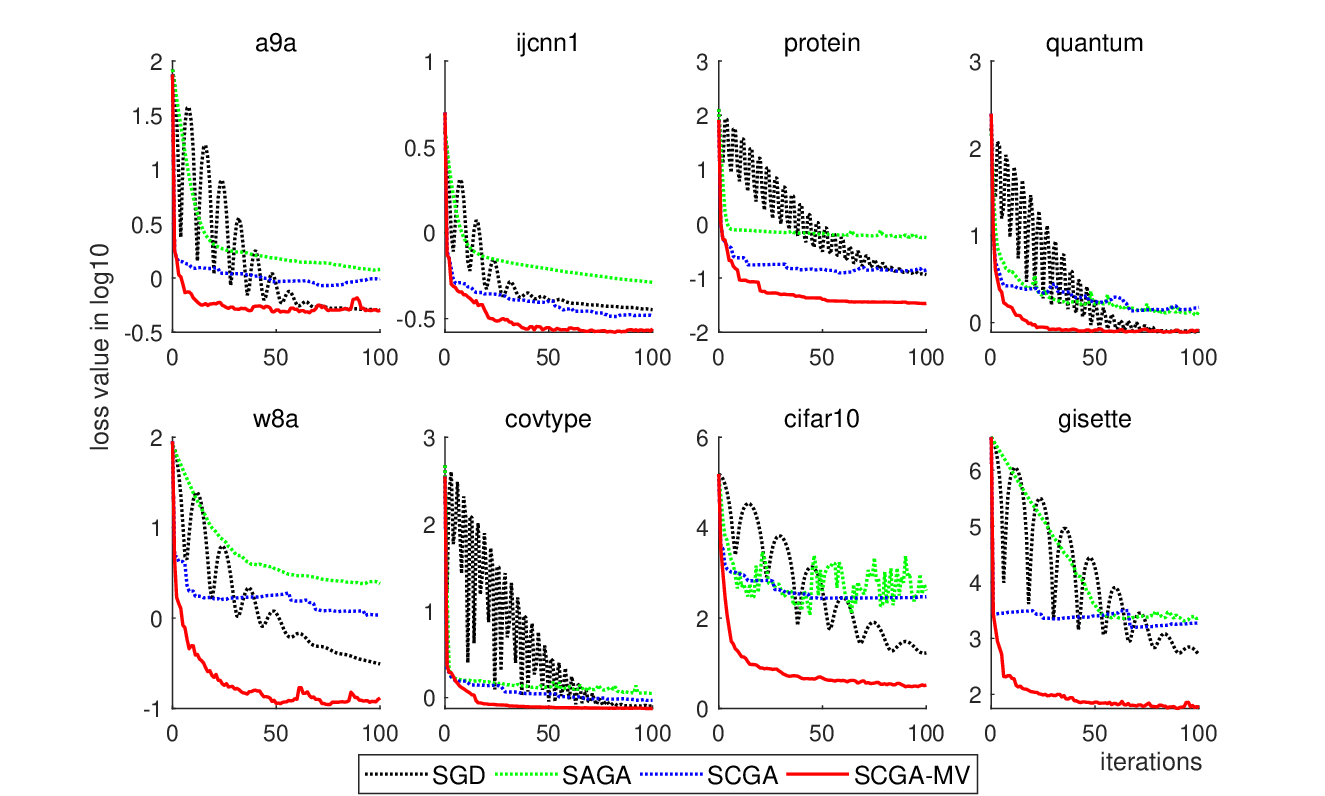}
	\caption{\revq{Performance of SGD, SAGA, SCGA, SCGA-MV algorithms on eight classification datasets (x-axis is times of iteration, y-axis is loss value in terms of $\log 10$)}}
	\label{fig:3}       
\end{figure}
\begin{figure}[H]
	\centering
	\includegraphics[width=0.95\textwidth]{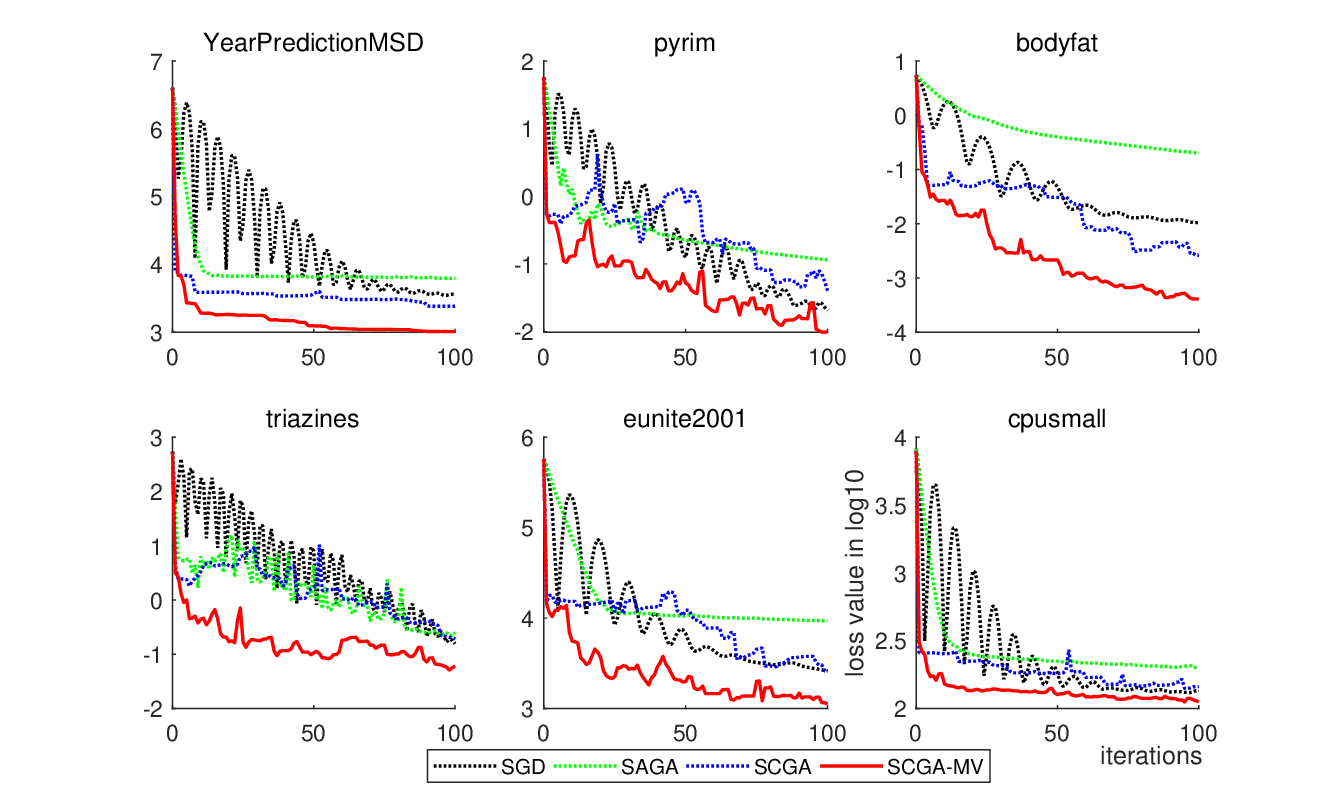}
	\caption{\revq{Performance of SGD, SAGA, SCGA, SCGA-MV algorithms on six  regression datasets(x-axis is times of iteration, y-axis is loss value in terms of log10)}}
	\label{fig:4}       
\end{figure}
\begin{figure}[H]
	\centering
	\includegraphics[width=0.95\textwidth]{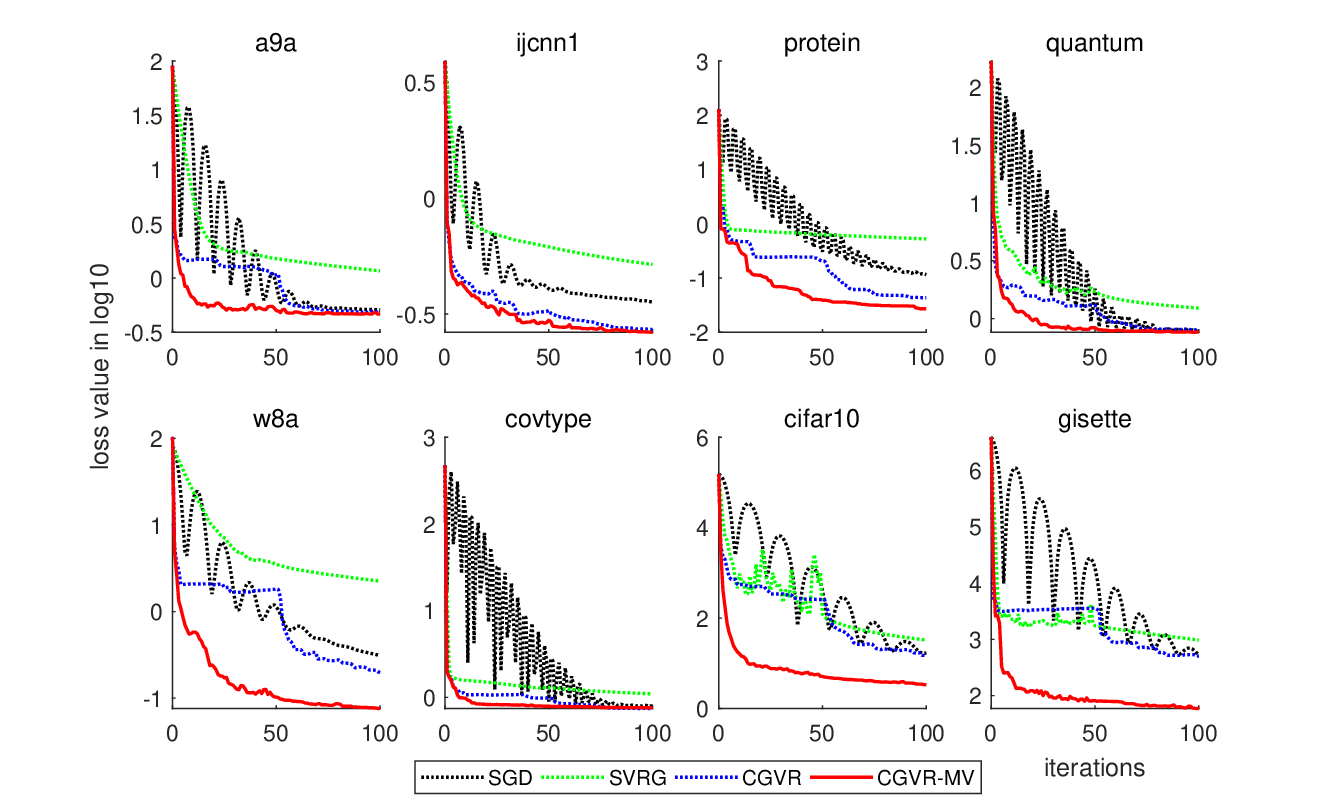}
	\caption{\revq{Performance of SGD, SVRG, CGVR, CGVR-MV algorithms on eight classification datasets (x-axis is times of iteration, y-axis is loss value in terms of log10)}}
	\label{fig:5}       
\end{figure}
\begin{figure}[H]
	\centering
	\includegraphics[width=0.95\textwidth]{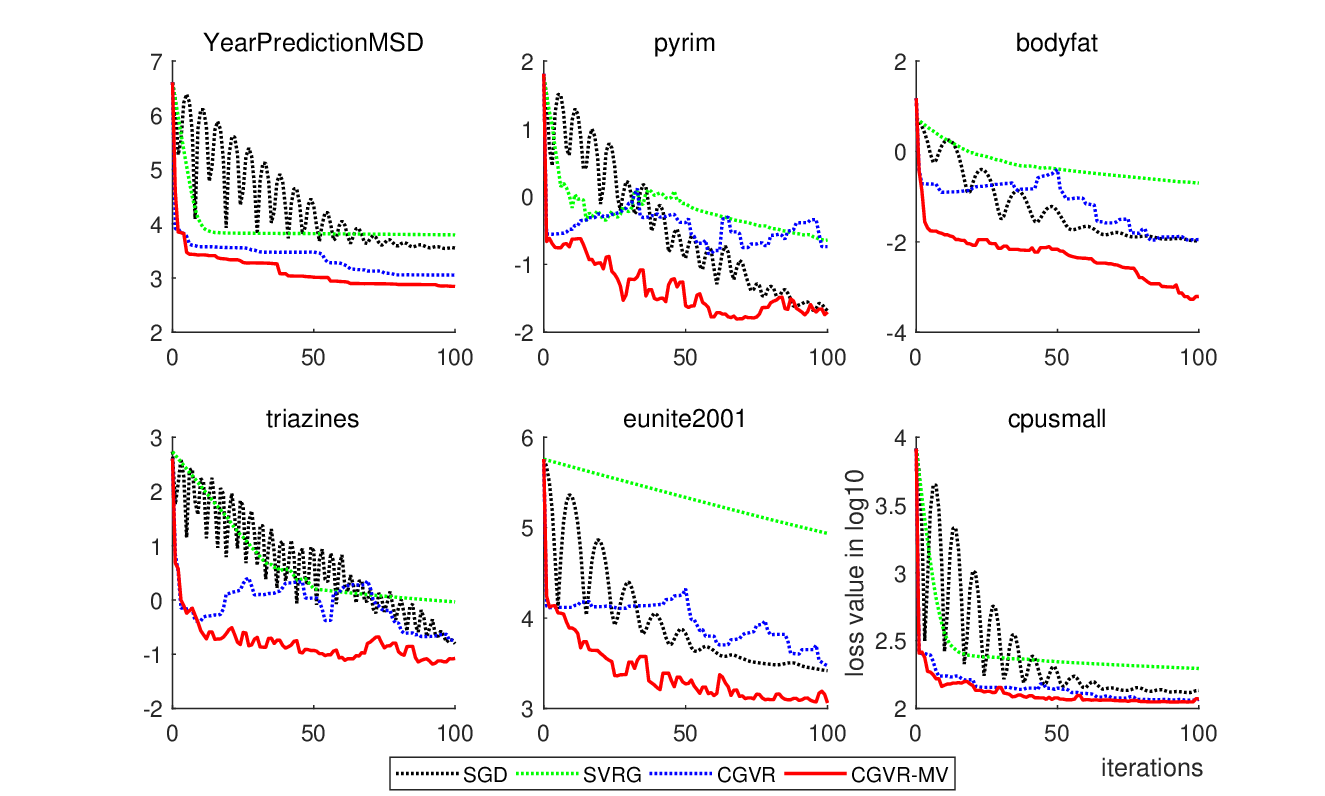}
	\caption{\revq{Performance of SGD, SVRG, CGVR, CGVR-MV on six  regression datasets (x-axis is times of iteration, y-axis is loss value in terms of log10)}}
	\label{fig:6}       
\end{figure}
\revq{We conduct numerical experiments to compare the two proposed algorithms, SCGA-MV and CGVR-MV, with their respective related algorithms.	
Since the proposed SCGA-MV is an enhanced variant of SCGA, which itself is a stochastic conjugate gradient algorithm based on SAGA, and SGD represents the most classical stochastic gradient algorithm, we include SGD, SAGA, and SCGA in our comparison with SCGA-MV. Similarly, for CGVR-MV, the comparison encompasses SGD, SVRG, and CGVR.	
	Fig.\ref{fig:3} and Fig.\ref{fig:4} present detailed comparisons between SCGA-MV and its counterparts SGD, SAGA, SCGA algorithms across both classification and regression datasets. Similarly, Fig.\ref{fig:5} and Fig. \ref{fig:6} focus specifically on evaluating CGVR-MV against its counterparts SGD, SVRG and CGVR on the same datasets. 	
	
	 The results shown in Fig.\ref{fig:3} and Fig.\ref{fig:5} show that two proposed SCGA-MV and CGVR-MV algorithms converge faster than the compared algorithms on classification datasets. Furthermore, Fig.\ref{fig:4} and Fig.\ref{fig:6} demonstrate that two proposed algorithms also excel on  regression datasets, exhibiting reduced volatility in line with the theoretical analysis presented in Section 2. The lower variance in the proposed algorithms contributes to their higher stability.}
%

\section{Conclusion}
\label{sec:5}

In this paper, we introduce a novel family of unbiased stochastic gradient estimators that encompasses existing estimators from the literature. Within this expanded family, we identify a specific estimator that not only maintains unbiasedness but also achieves minimal variance. Thanks to its being the best among the family, this specific estimator surpasses those utilized in SCGA and CGVR. We subsequently incorporate this new estimator into SCGA and CGVR, yielding two improved stochastic conjugate gradient algorithms with minimal variance: \revq{SCGA–MV} and \revq{CGVR–MV}. Additionally, we demonstrate the linear convergence rate of these new algorithms under conditions of strong convexity and smoothness. Ultimately, numerical experiments are conducted to compare the convergence rates of the improved algorithms with their original versions. The results reveal that \revq{SCGA–MV} and \revq{CGVR–MV} exhibit significant advantages over SCGA and CGVR, and demonstrate superior stability compared to the original algorithms.

In future research work, we aim to extend the application of our newly proposed estimator to a range of stochastic gradient-based algorithms, including stochastic gradient and stochastic conjugate gradient algorithms, and thoroughly investigate its performance within these algorithmic contexts. 


%
%



\end{document}